\documentclass[11pt]{amsart}
\usepackage[latin1]{inputenc}
\usepackage{amsmath,amsthm, amscd, amssymb, amsfonts}
\usepackage[dvips]{graphicx}
\usepackage[all]{xy}
\usepackage{leftidx}
\numberwithin{equation}{section}

\theoremstyle{plain}

\newtheorem{theorem}{Theorem}[section]
\newtheorem{corollary}[theorem]{Corollary}
\newtheorem{proposition}[theorem]{Proposition}
\newtheorem{lemma}[theorem]{Lemma}
\newtheorem{lemma-definition}[theorem]{Lemma-Definition}

\theoremstyle{definition}
\newtheorem{definition}[theorem]{Definition}

\newtheorem{example}[theorem]{Example}

\theoremstyle{remark}
\newtheorem{remark}[theorem]{Remark}

\newcommand{\C}{{\mathcal C}}
\newcommand{\E}{{\mathcal E}}

\newcommand{\Aa}{{\mathbb{A}}}

\newcommand{\D}{{\mathcal D}}

\newcommand{\K}{{\mathcal K}}
\newcommand{\M}{{\mathcal M}}

\newcommand{\Z}{{\mathcal Z}}
\newcommand{\Ss}{\mathbb{S}}
\newcommand\id{\operatorname{id}}
\newcommand\toto{{\longrightarrow}}
\newcommand\FPdim{\operatorname{FPdim}}
\newcommand\KER{\mathfrak{Ker}}
\newcommand\Aut{\operatorname{Aut}}
\newcommand\Irr{\operatorname{Irr}}

\newcommand\J{\mathcal{J}}

\newcommand\Hom{\operatorname{Hom}}
\newcommand\Opext{\operatorname{Opext}}
\newcommand\Rep{\operatorname{Rep}}

\newcommand\vect{\operatorname{Vect}}

\newcommand{\uno}{\textbf{1}}

\begin{document}

\title[Jordan-H\" older theorem for a class of fusion categories]{A Jordan-H\"
older theorem for weakly group-theoretical fusion categories}
\author{Sonia Natale}
\address{Facultad de Matem\'atica, Astronom\'\i a y F\'\i sica,
Universidad Nacional de C\'ordoba, CIEM -- CONICET, (5000) Ciudad
Universitaria, C\'ordoba, Argentina}
\email{natale@famaf.unc.edu.ar \newline \indent \emph{URL:}\/
http://www.famaf.unc.edu.ar/$\sim$natale}
\thanks{This work was partially supported by CONICET and Secyt (UNC)}
\date{Revised version of November 23, 2015}

\begin{abstract} We prove a version of the Jordan-H\" older theorem in the context of weakly group-theoretical fusion categories. This allows us to introduce the composition factors and the length of such a fusion category $\C$, which  are in fact Morita invariants of $\C$.  
\end{abstract}

\maketitle

\section{Introduction}

Weakly group-theoretical fusion categories were introduced in \cite{ENO2}. By definition, a fusion category $\C$ is weakly group-theoretical if it is (categorically) Morita equivalent to a nilpotent fusion category, that is, if there exist a nilpotent fusion category $\D$ and an idecomposable right module category $\M$ over $\D$ such that $\C$ is equivalent to the fusion category $\D^*_\M$ of $\D$-linear endofunctors of $\M$. 

We shall work over an algebraically closed base field $k$ of characteristic zero. Every weakly group-theoretical fusion category has integer Frobenius-Perron dimension. It is an open question whether any fusion category with this property is weakly group-theoretical \cite[Question 2]{ENO2}. In fact,  all known examples of fusion categories with integer Frobenius-Perron dimension are weakly group-theoretical; moreover, it is shown in \cite[Proposition 4.1]{ENO2} that the class of weakly group-theoretical fusion categories is stable under group 
extensions and  equivariantizations. It is also stable under taking Morita equivalent categories, tensor products,
Drinfeld center, fusion subcategories and components of quotient categories.

\medbreak The main result of this paper is an analogue of the Jordan-H\" older theorem from group theory for weakly group-theoretical fusion categories. 

We consider composition series of $\C$ defined as follows. In view of results of \cite{ENO2}, a fusion category $\C$ is weakly group-theoretical if and only if there exists a series of fusion categories
\begin{equation}\label{cs-wgt-1}\vect = \C_0, \C_1, \dots , \C_n = \C,\end{equation}
such that for all $1 \leq i \leq n$, the Drinfeld center $\Z(\C_i)$ contains a Tannakian subcategory $\E_i$ and the de-equivariantization of the M\" uger centralizer $\E_i'$ in $\Z(\C_i)$ by $\E_i$ is equivalent to $\Z(\C_{i-1})$ as braided fusion categories. (See Section \ref{prels} below for the unexplained notions involved in this characterization.)

\medbreak We call such a series \eqref{cs-wgt-1} a \emph{central series} of the weakly group-theoretical fusion category $\C$. Since the categories $\E_i$ are Tannakian, then, for all $i = 1, \dots, n$, there exist finite groups $G_1, \dots, G_n$, such that $\E_i \cong \Rep G_i$ as symmetric fusion categories. The groups $G_1, \dots, G_n$ are called the \emph{factors} of the series \eqref{cs-wgt-1}. We call two central series \emph{equivalent} if they have the same (up to isomorphisms) factors counted with multiplicities.

We define a \emph{composition series} of $\C$ to be a central series whose factors are simple groups.   

\medbreak Our main result, Theorem \ref{main}, states that any two composition series of a weakly group-theoretical fusion category are equivalent.  The  proof of Theorem \ref{main} is given in Section \ref{jordan-hoelder}; it relies on the structure of a crossed braided fusion category, a notion introduced by Turaev \cite{turaev},  in the de-equivariantization of a braided fusion category by a Tannakian subcategory.

\medbreak Theorem \ref{main} allows us to introduce the \emph{composition factors} and the \emph{length} of $\C$ as the factors and length of any composition series of $\C$. By definition, the composition factors and the length are Morita invariants of $\C$. 

We discuss these invariants in some examples in Section \ref{examples}, including group extensions and equivariantizations of weakly group-theoretical fusion categories, group-theoretical fusion categories, nilpotent Hopf algebras and Drinfeld centers. In particular, we show that for group-theoretical Hopf algebras, the composition factors considered in this paper do not coincide, in general, with  those studied in \cite{jh-hopf} in the context of Hopf algebra extensions (see Subsection \ref{gttic}).

\section{Preliminaries}\label{prels}

A fusion category over $k$ is a semisimple tensor category over $k$ with finitely many
isomorphism classes of simple objects. We refer the reader to \cite{ENO}, \cite{DGNOI} for the main features and facts about fusion categories  implicitly used throughout this paper.

\medbreak Let $\C$ be a fusion category over $k$. A fusion subcategory of $\C$ is a full replete tensor subcategory $\D$ of $\C$
stable under taking direct summands. 

A \emph{tensor functor} between fusion categories $\C$ and $\D$  is a $k$-linear strong monoidal functor $F: \C \to \D$.

\medbreak Let $X$ be an object of $\C$. The \emph{Frobenius-Perron dimension} of $X$, denoted $\FPdim X$, is the Frobenius-Perron eigenvalue of the matrix of left multiplication by the class of $X$ in the Grothendieck ring of $\C$ in the basis $\Irr(\C)$ consisting of isomorphism classes of simple objects. The Frobenius-Perron dimension of $\C$ is defined as $\FPdim \C = \sum_{X \in \Irr(\C)}\FPdim X$. 

An object $X$ of $\C$ is \emph{invertible} if $X \otimes X^* \cong \uno \cong X^* \otimes X$. Thus $X$ is invertible if and only if $\FPdim X = 1$.

\medbreak We recall for later use that for objects $X, Y, Z$ of $\C$ there are canonical isomorphisms
\begin{equation}\Hom(X, Y\otimes Z) \cong \Hom(Y^*, Z\otimes X^*) \cong \Hom(Y, X\otimes Z^*).\end{equation}

\subsection{Group extensions of fusion categories}
Let $G$ be a finite group and let $\C$ be a fusion category. A $G$-\emph{grading} on $\C$ is a decomposition \begin{equation}\label{oplus}\C = \oplus_{g\in G} \C_g,\end{equation} such that $\C_g \otimes \C_h \subseteq \C_{gh}$, for all $g, h \in G$. Alternatively, a $G$-grading on $\C$ can be defined as a map $\gamma:\Irr(\C) \to G$ such that $\gamma(Z) = \gamma(X)\gamma(Y)$, for all simple objects $X, Y, Z$ of $\C$ such that $Z$ is a constituent of $X \otimes Y$. See \cite[Section 2.3]{DGNOI}.

If \eqref{oplus} is a $G$ grading on $\C$, then 
$\C_g^* \subseteq \C_{g^{-1}}$, for all $g \in G$, and the neutral component $\C_e$ is a fusion subcategory of $\C$.

The grading \eqref{oplus} is called \emph{faithful} if $\C_g \neq 0$, for all $g \in G$. Two group gradings $\gamma_1:\Irr(\C) \to G_1$ and $\gamma_2:\Irr(\C) \to G_2$ are called \emph{equivalent} if there exists a group isomorphism $f: G_1 \to G_2$ such that $\gamma_2 = f \gamma_1$.


\medbreak The fusion category $\C$ is called a \emph{$G$-extension} of a
fusion category $\D$ if there is a faithful grading $\C = \oplus_{g\in G} \C_g$
with neutral component $\C_e \cong \D$. If this is the case, then $\FPdim \C_g = \FPdim \D$, for all $g \in G$. In particular, $\FPdim \C = |G|\FPdim \D$.

\medbreak If $\C$ is any fusion category, there exist a finite group $U(\C)$, called the
\emph{universal grading group} of $\C$, and a canonical faithful grading $\C =
\oplus_{g \in U(\C)}\C_g$, whose neutral component is the fusion subcategory  $\C_{ad}$, generated by the objects $X\otimes
X^*$, where $X$ runs over the simple objects of $\C$. The category  $\C_{ad}$ is called the \emph{adjoint}
subcategory of $\C$. See \cite{gel-nik}.

\medbreak It is shown in \cite[Corollary 2.6]{DGNOI} that there is a bijective correspondence between equivalence classes of
faithful gradings of $\C$ and fusion subcategories $\D$ of $\C$ containing $\C_{ad}$ such that the
subgroup $G_\D = \{g \in U_\C:\, \D\cap \C_g \neq 0\}$ is a normal subgroup of the universal grading group $U_\C$. 

This correspondence associates to every faithful grading of $\C$ its trivial component $\D = \C_e$, and to every fusion subcategory $\D$ containing $\C_{ad}$ with $G_\D$ normal in $U_\C$, the $U_\C/G_\D$-grading $\C = \oplus_{t \in U_\C/G_\D} \C_{t}$, where $\C_{t} = \oplus_{g \in t}\C_g$. In particular, $\D = \oplus_{g \in G_\D} \C_g$.

\medbreak Note that, if $\D_1$ and $\D_2$ are two fusion subcategories of $\C$ containing $\C_{ad}$, then  $\D_1 \subseteq \D_2$ if and only if $G_{\D_1} \subseteq G_{\D_2}$. 
Therefore, as a consequence of \cite[Corollary 2.6]{DGNOI}, we obtain:

\begin{proposition}\label{normal-gding} Let $\C$ be a fusion category endowed with a faithful grading $\C = \oplus_{g \in G}\C_g$. Then there is a bijective correspondence between equivalence classes of faithful gradings of $\C$ with neutral component $\D$ such that $\C_e \subseteq \D$ and normal subgroups of $G$. This correspondence associates to every normal subgroup $N$ the $G/N$-grading $\C = \oplus_{t \in G/N} \C_{t}$, where $\C_{t} = \oplus_{g \in t}\C_g$. \qed \end{proposition}

\begin{remark}\label{rmk-normal-gding} Observe that, for every normal subgroup $N$ of $G$, the neutral component $\D$ of the associated $G/N$-grading of $\C$ in Proposition \ref{normal-gding}, is  $\D = \oplus_{g \in N}\C_g$. In particular, $\D$ is an $N$-extension of $\C_e$. \end{remark}

\subsection{Normal tensor functors and equivariantization}

Let $\C, \D$ be fusion categories over $k$ and let $F:\C \to \D$ be a tensor functor. We shall denote by $F(\C)$ the fusion subcategory  of $\D$ generated by the essential image of $F$.  Thus $F(\C)$ is generated as an additive category by the objects $Y$ of $\D$ which are subobjects of $F(X)$, for some object $X$ of $\C$.

\medbreak The functor $F$ is called \emph{dominant} if every object $Y$ of $\D$ is a subobject of $F (X)$ for some object $X$ of $\C$. Observe that, if $F: \C \to \D$ is any tensor functor between fusion categories $\C, \D$, then the corestriction of $F$  is a dominant tensor functor   $\C \to F(\C)$.

\medbreak We shall say that an object of a fusion category is \emph{trivial} if it is isomorphic to a direct sum of copies of the unit object. Let $\KER_F$ denote the full tensor subcategory of $\C$ consisting of all objects $X$ such that $F(X)$ is a trivial object of $\D$. Normal tensor functors between tensor categories were defined in \cite[Definition 3.4]{tensor-exact}. Since $\C$ and $\D$ are fusion categories, a tensor functor $F: \C \to \D$ is \emph{normal} if and only if for every simple object
$X \in \C$ such that $\Hom(\uno, F (X)) \neq  0$, we have $X \in \KER_F$. See \cite[Proposition 3.5]{tensor-exact}.

\begin{lemma}\label{nd-equiv} Let $F: \C \to \D$ be a normal tensor functor, where $\C$ and $\D$ are fusion categories. Suppose that $\E \subseteq \C$ is a tensor subcategory such that $\E \cap \KER_F  = \vect$. Then $F$ induces by restriction an equivalence of tensor categories $F\vert_{\E}: \E \to F(\E)$.
\end{lemma}

\begin{proof} The restriction $F\vert_{\E}: \E \to F(\E)$ is a dominant normal tensor functor whose kernel is $\KER_{F\vert_{\E}} = \E \cap \KER_F$. Hence the assumption implies that $\KER_{F\vert_{\E}}$ is the trivial fusion subcategory of $\E$. 
By \cite[Lemma 3.6 (2)]{tensor-exact}, the functor $F\vert_{\E}$ is full. Being dominant, $F\vert_{\E}$ is an equivalence in view of \cite[Lemma 3.6 (1)]{tensor-exact}. 
\end{proof}

Let $G$ be a finite group and let $\C$ be a fusion category. An \emph{action} of $G$ on $\C$ by tensor autoequivalences is a monoidal functor $\rho: \underline
G \to \underline \Aut_{\otimes} \, \C$, where $\underline G$ is the strict monoidal category whose objects are the elements of $G$, morphisms are identitites and tensor product is given by the multiplication of $G$, and   $\underline{\Aut}_{\otimes}(\mathcal{C})$ is the monoidal category of tensor autoequivalences of $\mathcal{C}$.  

Let $\rho_2^{g,h}:\rho^g\rho^h\rightarrow\rho^{gh}$, $g, h\in G$, and $\rho_0: \id_\C\rightarrow\rho^e$, denote the monoidal structure of the functor $\rho$.

\medbreak The \emph{equivariantization} of $\C$ with respect to the action $\rho$ is the fusion category $\C^G$, whose objects are pairs  $(X, \mu)$, such that $X$
is an object of $\C$ and $\mu = (\mu^g)_{g \in G}$, is a collection of
isomorphisms $\mu^g:\rho^gX \to X$, $g \in G$, satisfying the conditions
$$\mu^g\rho^g(\mu^h)=\mu^{gh}(\rho_2^{g,h})_X, \hspace{1.5cm} \mu_e\rho_{0X} =\id_X,$$ for all $g,h\in G$, $X \in \C$, and morphisms $(X, \mu) \to (X',\mu')$ are morphisms $f: X \to X'$ in $\C$ such that $f\mu^g = {\mu'}^g\rho^g(f)$, for all $g\in G$. 
The tensor product $(X, \mu) \otimes (X',\mu')$ in $\C^G$ is defined by $(X, \mu) \otimes (X',\mu') = (X \otimes X', (\mu^g\otimes{\mu'}^g)\rho_2^g)$, where, for all $g \in G$, $\rho_2^g: \rho^^g(X \otimes X') \to \rho^g(X) \otimes \rho^g(X')$ is the isomorphism given by the monoidal structure of $\rho^g$.

\medbreak The category $\Rep G$ embeds into $\C^G$ as a fusion subcategory and the forgetful functor $F: \C^G \to \C$ is a normal dominant tensor functor with kernel $\KER_F \cong \Rep G$.  See \cite[Subsection 5.3]{tensor-exact}.

\subsection{Braided fusion categories and  group crossed braided fusion categories}

A \emph{braided fusion category} is a fusion category $\C$ endowed with a natural isomorphism $c: \otimes \to \otimes^{op}$, called a \emph{braiding}, that satisfies the hexagon axioms. 

A braided fusion category $\C$ is called \emph{symmetric} if $c_{Y, X}c_{X, Y} = \id_{X\otimes Y}$, for all objects $X, Y$ of $\C$. 

If $G$ is a finite group, then the category $\Rep G$, endowed with the braiding given by the flip isomorphism, is a symmetric fusion category. A \emph{Tannakian fusion category} is a symmetric fusion category equivalent to $\Rep G$ as a braided fusion category.   

\medbreak Let $\C$ be a braided fusion category. Two objects $X, Y$ of $\C$ are said to \emph{centralize each other} if $c_{Y, X}c_{X, Y} = \id_{X\otimes Y}$. The \emph{M\" uger centralizer} of a fusion subcategory $\D$ is the fusion subcategory $\D'$ of $\C$ whose objects $Y$ centralize every object of $\D$. So that, the fusion  subcategory $\D$ is symmetric if and only if $\D \subseteq \D'$.

\medbreak Let $G$ be a finite group. Recall that a \emph{braided $G$-crossed fusion
category} is a fusion category $\D$ endowed with a $G$-grading $\D
= \oplus_{g \in G}\D_g$, an action of $G$ by tensor autoequivalences
$\rho:\underline G \to \underline \Aut_{\otimes} \, \D$, such that $\rho^g(\D_h)
\subseteq
\D_{ghg^{-1}}$, for all $g, h \in G$, and a $G$-braiding $c: X \otimes Y \to
\rho^g(Y) \otimes X$, $g \in G$, $X \in \D_g$, $Y \in \D$, subject to
appropriate compatibility conditions.

\medbreak Suppose that $\E \cong \Rep G$ is a Tannakian subcategory of a braided fusion category $\C$. 
The algebra $k^G$ of functions on $G$ endowed with the regular action of $G$ corresponds to a semisimple commutative  algebra $A$ in $\C$ such that $\Hom(\textbf{1}, A) \cong k$.

The de-equivariantization of $\C$ with respect to $\E$, denoted $\C_G$, is the fusion category $\C_A$ of right $A$-modules in $\C$. The category $\C_G$ is a braided $G$-crossed fusion category in a canonical way and we have $\C \cong (\C_G)^G$. Under this equivalence, the forgetful functor $(\C_G)^G \to \C_G$ becomes identified with the canonical functor $F: \C \to \C_G$, $F(X) = X \otimes A$.

\medbreak The neutral component of $\C_G$ with respect to the associated $G$-grading will be denoted by $\C_G^0$. Then $\C_G^0$ is a braided fusion category and the crossed action of $G$ on $\C_G$ induces an action of $G$ on $\C_G^0$ by braided auto-equivalences.  The equivariantization $(\C_G^0)^G$  coincides with the M\" uger centralizer $\E'$ of the
Tannakian subcategory $\E$ in $\C$, and there is an equivalence of braided fusion categories
$$\C_G^0 \cong \E'_G.$$ 
See \cite[Proposition 3.24]{mueger-crossed}. In particular, the canonical tensor functor $F: \C \to \C_G$ restricts to a dominant normal braided tensor functor $$F\vert_{\E'}: \E' \to \C_G^0.$$

\medbreak Conversely, every $G$-crossed braided fusion category gives  rise, through the equivariantization process, to a braided fusion category containing $\Rep G$ as a Tannakian subcategory. 

Hence the de-equivariantization and equivariantization procedures define inverse bijections
between equivalence classes of braided fusion categories containing $\Rep G$ as a
Tannakian subcategory and equivalence classes of $G$-crossed braided fusion categories
\cite{mueger-crossed}, \cite[Section 4.4]{DGNOI}.

\begin{lemma}\label{normal-h} Let $\D$ be a $G$-crossed braided fusion category. Every $G$-stable fusion subcategory $\E$ of $\D$ determines canonically a normal subgroup $H$ of $G$ such that $h \in H$ if and only if $\E \cap \D_h \neq 0$. The fusion category $\E$ is faithfully graded by $H$ with neutral component $\E \cap \D_e$.
\end{lemma}

\begin{proof} Let $H$ be the set of all elements $h \in G$ such that $\E \cap \D_h \neq 0$. Since the tensor product of nonzero objects is a nonzero object, then $H$ is a subgroup of $G$. Suppose $h \in H$, and let $0 \neq X \in \E \cap \D_h$. Since $\E$ is $G$-stable by assumption, we find that, for all $g \in G$, $0 \neq \rho^g(X) \in \E\cap \D_{ghg^{-1}}$. Therefore $ghg^{-1} \in H$, for all $g \in G$, and $H$ is a normal subgroup of $G$.  

It follows from the definition of $H$ that there is a faithful grading $\E = \bigoplus_{h \in H}\E_h$, where $\E_h = \E \cap \D_h$. This finishes the proof of the lemma.  \end{proof}

\begin{proposition}\label{simple-gstable} Let $\D$ be a $G$-crossed braided fusion category. Suppose that $\E \subseteq \D$ is a $G$-stable fusion subcategory such that $\E$ contains no proper nontrivial fusion subcategories.  Then either $\E \subseteq \D_e$ or $\E$ is pointed of prime dimension and the group $G(\E)$ of invertible objects of $\E$ is isomorphic to a normal subgroup of $G$.
\end{proposition}

\begin{proof} Let $H$ be the normal subgroup of $G$ attached to $\E$ by Lemma \ref{normal-h}, so that the category $\E$ is faithfully $H$-graded with neutral component $\E\cap \D_e$. Since by assumption $\E$ contains no proper nontrivial fusion subcategories, we must have $\E\cap \D_e = \E$ or $\E\cap \D_e = \vect$. If the first possibility holds, then $\E\subseteq \D_e$. Suppose that the second possibility holds and consider the faithful grading $\E = \bigoplus_{h \in H}\E_h$. Since $\FPdim \E_h = \FPdim \E_e = 1$, for all $h \in H$, it follows that in this case $\E$ must be pointed and its group of invertible objects is isomorphic to $H$. Furthermore, since $\E$ contains no proper nontrivial fusion subcategories, then $H$ must be cyclic of prime order. This proves the proposition. \end{proof}

\section{Weakly group-theoretical fusion categories}\label{wgt}

\subsection{Nilpotent fusion categories}\label{nilpotent} Let $\C$ be a fusion category.  Recall from \cite{gel-nik} that the \emph{upper central series} of $\C$ is defined recursively by $\C^{(0)} = \C$ and $\C^{(j)} = (\C^{(j-1)})_{ad}$, for all $j \geq 1$. The fusion category $\C$ is called \emph{nilpotent} if its upper central series converges to $\vect$, that is, if there exists $n \geq 0$ such that $\C^{(n)} \cong \vect$.

Equivalently, $\C$ is nilpotent if and only if there exists a series of fusion subcategories 
\begin{equation}\label{cs-nil}\vect = \C_0 \subseteq \C_{1} \subseteq \dots \subseteq \C_n = \C,\end{equation}
such that $\C_i$ is a $G_i$-extension of $\C_{i-1}$, for all $i = 1, \dots, n$, for some finite groups $G_1, \dots, G_n$.

\begin{definition}\label{def-csnil} The finite groups $G_1, \dots, G_n$ will be called the \emph{factors} of the series \eqref{cs-nil}. 
A \emph{refinement} of 
\eqref{cs-nil} is a series of fusion subcategories
\begin{equation}\label{refi}\vect = \C_0' \subseteq \C_{1}' \subseteq \dots \subseteq \C_m' = \C,\end{equation} where $\C_i'$ is a $G_i'$-extension of $\C_{i-1}'$, for all $i = 1, \dots, m$, for some finite groups $G_1', \dots, G_m'$,
such that for all  $j = 1, \dots, n-1$, there exists $0 < N_j < m$, with $N_1 <
N_2 < \dots < N_m$, and $\C_j = \C'_{N_j}$.  If  \eqref{refi} is a refinement of
\eqref{cs-nil} and it does not coincide
with \eqref{cs-nil}, we shall say that it is a  \emph{proper refinement}. If \eqref{cs-nil} admits no proper refinement, we shall say it is a \emph{composition series} of $\C$.
\end{definition}

\begin{lemma}\label{f-simple} Let $\C$ be a nilpotent fusion category and suppose $\vect = \C_0 \subseteq \C_{1} \subseteq \dots \subseteq \C_n = \C$ is a series of fusion subcategories of $\C$ as in \eqref{cs-nil} with factors $G_1, \dots, G_n$. Then the following are equivalent:
\begin{itemize}\item[(i)] $\vect = \C_0 \subseteq \C_{1} \subseteq \dots \subseteq \C_n = \C$ is a composition series of $\C$.
\item[(ii)] The factors $G_1, \dots, G_n$ are simple groups.
\end{itemize}   \end{lemma}

\begin{proof} Let $1\leq i \leq n$. Since $\C_i$ is a $G_i$-extension of $\C_{i-1}$, it follows from Proposition \ref{normal-gding} that there is a bijective correspondence between fusion subcategories $\D$ such that $\C_{i-1} \subseteq \D \subseteq \C_i$ and $\C_i$ is an extension of $\D$, and normal subgroups $N$ of $G_i$. Moreover, every such fusion subcategory $\D$ is an extension of $\C_{i-1}$ (see Remark \ref{rmk-normal-gding}). This implies the lemma.
\end{proof}

\begin{remark} Lemma \ref{f-simple} and Proposition \ref{normal-gding} imply that every nilpotent fusion category $\C$ has a composition series. \end{remark}

\begin{example}\label{pointed} Let $G$ be a finite group and let $\omega \in H^3(G, k^*)$. Consider the fusion  category $\C = \C(G, \omega)$ of finite dimensional $G$-graded vector spaces with associativity induced by the $3$-cocycle $\omega$. Thus, $\C(G, \omega)$ is a pointed fusion category whose group of invertible objects is isomorphic to $G$, and it is a nilpotent fusion category. 

Let $\{e\} = G_{(0)} \subseteq G_{(1)} \subseteq \dots \subseteq G_{(n)} = G$ be a composition series of $G$ with factors $G_i = G_{(i)}/G_{(i-1)}$, $i = 1, \dots, n$. Then the series 
$$\vect = \C_0 \subseteq \C_{1} \subseteq \dots \subseteq \C_n = \C,$$
where $\C_i = \C(G_{(i)}, \omega\vert_{G_{(i)}})$, $0 \leq i \leq n$, is a composition series of $\C$ with factors $G_1, \dots, G_n$. \end{example}

\subsection{Weakly group-theoretical fusion categories} A (right) \emph{module category} over a fusion category $\C$ is a finite semisimple $k$-linear abelian category $\M$ endowed with a bifunctor $\otimes: \M \times \C \to \M$ satisfying the associativity and unit axioms for an action, up to coherent natural isomorphisms. 

The module category $\M$ is called \emph{indecomposable}
if it is not equivalent as a module category to a direct sum of non-trivial module categories. If $\M$ is an indecomposable module category over $\C$, then the category $\C^*_{\mathcal M}$ of $\C$-module endofunctors of $\M$ is also a fusion category. 

Two fusion categories $\C$ and $\D$
are \emph{Morita equivalent} if $\D$ is equivalent to 
$\C^*_{\mathcal M}$ for some indecomposable module
category $\mathcal M$. By \cite[Theorem 3.1]{ENO2}, the fusion categories $\C$ and $\D$ are Morita equivalent if and only if its Drinfeld centers are equivalent as braided fusion categories.

\medbreak A  fusion category $\C$ is called \emph{group-theoretical} if it is Morita equivalent to a pointed fusion category, that is, to a fusion category all of whose simple objects are invertible. More generally, $\C$ is called \emph{weakly group-theoretical} if it is Morita equivalent to a nilpotent  fusion category, and it is called solvable if it is Morita equivalent to a cyclically nilpotent fusion category. 

It is well-known that every pointed fusion category is equivalent to the category $\C(G, \omega)$, for some finite group $G$ (isomorphic to the group of invertible objects of $\C$) and $\omega \in H^3(G, k^*)$. Hence every group-theoretical fusion category is Morita equivalent to some of the fusion categories $\C(G, \omega)$.

\medbreak Let $G$ be a finite group. By \cite[Theorem 1.3]{ENO2}, a fusion category $\C$ is Morita equivalent to a $G$-extension of a fusion category  $\D$ if and only if $\Z(\C)$ contains a
Tannakian subcategory $\E \cong \Rep(G)$ such that the de-equivariantization of the M\" uger centralizer $\E'$ by $\E$ is equivalent
to $\Z(\D)$ as a braided fusion categories.

Therefore, if $\C$ is a fusion category, $\C$ is weakly group-theoretical if and only if there exists a series of fusion categories
\begin{equation}\label{cs-wgt}\vect = \C_0, \C_1, \dots , \C_n = \C,\end{equation}
such that for all $1 \leq i \leq n$, the Drinfeld center $\Z(\C_i)$ contains a Tannakian subcategory $\E_i$ and the de-equivariantization of the M\" uger centralizer $\E_i'$ in $\Z(\C_i)$ by $\E_i$ is equivalent to $\Z(\C_{i-1})$ as braided fusion categories. See also \cite[Proposition 4.2]{ENO2}. 

Observe that, for every $i = 1, \dots, n$, there exists a finite group $G_i$ such that $\E_i \cong \Rep G_i$ as symmetric fusion categories.

\begin{definition}\label{def-cs-wgt} A series \eqref{cs-wgt} will be called a \emph{central series} of $\C$. The number $n$ will be called the \emph{length} of the central series \eqref{cs-wgt}. 

 The groups $G_1, \dots, G_n$ will be called the \emph{factors} of the series \eqref{cs-wgt}.

 Let $\C$ be a weakly group-theoretical fusion category. Two central series $\vect = \C_0, \C_1, \dots , \C_n = \C$ and $\vect = \D_0, \D_1, \dots , \D_m = \C$ of $\C$ will be called \emph{equivalent} if there is a bijection $\{0, \dots, n\}\to  \{0, \dots,  m\}$ such that the corresponding factors are isomorphic.

The central series \eqref{cs-wgt} will be called a \emph{composition series} of $\C$ if the factors are simple groups. 
\end{definition}

We shall show in the next section that any two composition series of a weakly group-theoretical fusion category $\C$ are equivalent. This will allow us to introduce the composition factors of $\C$, which are a Morita invariant of $\C$.

\begin{remark} Suppose that $\C$ is a nilpotent fusion category. As a consequence of the above mentioned Theorem 1.3 of \cite{ENO2}, every series of fusion subcategories $\vect = \C_0 \subseteq \C_{1} \subseteq \dots \subseteq \C_n = \C$ of $\C$ as in \eqref{cs-nil}  with factors $G_1, \dots, G_n$, induces a central series $\vect = \C_0, \C_{1}, \dots, \C_n = \C$ of $\C$ with factors $G_1, \dots, G_n$. 

Furthermore, in view of Lemma \ref{f-simple}, a series of fusion subcategories  as in \eqref{cs-nil} is a composition series of $\C$ in the sense of Definition \ref{def-csnil}, if and only if the induced central series  is a composition series of $\C$ in the sense of Definition \ref{def-cs-wgt}. 

Thus the definition of a composition series of a weakly group-theoretical fusion category extends that of a composition series of a nilpotent fusion category given in Subsection \ref{nilpotent}. 
\end{remark}

\section{Jordan-H\" older theorem}\label{jordan-hoelder}

\begin{lemma}\label{dimp} Let $\C$ be a braided fusion category. Let $p$ be a prime number and suppose that $\E \cong \Rep G$ and $\K \cong \Rep L$ are Tannakian subcategories of $\C$ of dimension $p$ such that $\K \cap \E = \vect$ and $\K \cap \E' = \vect$. Then there exists an equivalence of braided fusion categories
$$\C_G^0 \cong \C_L^0.$$ \end{lemma}

\begin{proof} Let $\E \vee \K$ denote the fusion subcategory of $\C$ generated by $\E$ and $\K$. Since both $\E$ and $\K$ are pointed fusion categories of dimension $p$ and $\E \cap \K = \vect$, then $\E \vee \K$ is a pointed fusion category of dimension $p^2$. The assumption $\K \cap \E' = \vect$ implies that $\E \vee \K$ is not symmetric. Therefore $\E$ and $\K$ are maximal among Tannakian subcategories of $\E \vee \K$. 

As a consequence of \cite[Theorem 5.11]{DGNOI}, we find  that there is an equivalence of braided fusion categories $\E'_G \cong \K'_L$. This implies the lemma, since $\C_G^0  \cong \E'_G$ and $\C_L^0 \cong \K'_L$ as braided fusion categories. \end{proof}

\begin{proposition}\label{br-equiv} Let $\C$ be a braided fusion category. Suppose $\E \cong \Rep G \subseteq \C$ is a Tannakian subcategory and let $F: \C \to \C_G$ be the canonical tensor functor. Let $\D$ be a fusion subcategory of $\C$ such that $\D\cap \E = \vect$. Then $F$ induces by restriction an equivalence of fusion categories $F\vert_{\D}: \D\to F(\D)$.

Assume in addition that $F(\D) \subseteq \C_G^0$. Then $\D \subseteq \E'$ and $F\vert_{\D}: \D\to F(\D)$ is an equivalence of braided fusion categories.  \end{proposition}

\begin{proof} The functor $F$ is a normal dominant tensor functor whose kernel is $\KER_F = \E$. By Lemma \ref{nd-equiv}, the restriction $F\vert_{\D}: \D\to F(\D)$ is an equivalence of fusion categories.  

Let us now assume that $F(\D) \subseteq \C_G^0$. Since the restriction $F\vert_{\E'}$ is a braided tensor functor, it will be enough to show that $\D\subseteq \E'$. To see this, we consider the equivalence $F\vert_{\D}: \D\to F(\D)$ guaranteed by the first part of the proposition. 

Let $X$ be a simple object of $\D$. Then $F(X)$ is a simple $G$-equivariant object of $\C_G^0$.  Since the restriction $F\vert_{\E'}: \E' \to \C_G^0$ is dominant, there exists an object $Y$ of $\E'$ such that $F(X)$ is a subobject of $F(Y)$. Therefore $0 \neq \Hom(\uno, F(Y \otimes X^*))$ and it follows that  $\Hom(\uno, F(Z)) \neq 0$ for some simple constituent of $Y \otimes X^*$ in $\C$. Since the functor $F$ is normal, then $F(Z)$ is a trivial object of $\C_G$, that is, $Z \in \KER_F = \E$. 

On the other hand, $\Hom(X, Z^* \otimes Y) \cong \Hom(Z, Y \otimes X^*) \neq 0$. Then $X$ is a simple constituent of $Z^* \otimes Y$ and therefore $X \in \E'$, because $Y \in \E'$ and $Z^* \in \E \subseteq \E'$. This finishes the proof of the proposition.  \end{proof}

\begin{lemma}\label{ittd} Let $\C$ be a braided fusion category and let $G_1$ and $L$ be finite groups. Suppose $\C$ contains Tannakian subcategories $\E \subseteq \J$, with $\J \cong \Rep G_1$, $\E \cong \Rep G_1/L$. Let $G = G_1/L$. Then $\C^0_{G}$ contains a Tannakian subcategory equivalent to $\J_G \cong \Rep L$ and there is an equivalence of braided fusion categories
$$\C_{G_1}^0 \cong (\C^0_{G})^0_{L}.$$
\end{lemma}

\begin{proof}  By \cite[Corollary 4.31]{DGNOI}, $\J_G \cong \Rep L$ is a Tannakian subcategory of $\E'_G \cong \C^0_G$. Moreover, since $\J'$ is a braided fusion category over $G_1$, then \cite[Lemma 4.32]{DGNOI} gives an equivalence of braided fusion categories 
\begin{equation}\label{i}(\J'_G)_{L} \cong \J'_{G_1}.\end{equation}

On the other hand, $\E'$ is a braided fusion category over $\E$ containing $\J$. 
Observe that $\J' \subseteq \E'$ and therefore $\J'$ coincides with the centralizer of $\J$ in $\E'$. Since $\E \subseteq \J$, it follows from \cite[Proposition 4.30 (iii)]{DGNOI} that 
\begin{equation}\label{ii}\J'_G = (\J_G)',\end{equation}
as fusion subcategories of $\E'_G \cong \C^0_G$.

Combining \eqref{i} and \eqref{ii}, we get an equivalence of braided fusion categories 
$$(\J_G)'_{L} \cong \J'_{G_1}.$$ This implies the lemma, since $(\C^0_G)^0_{L} \cong (\J_G)'_{L}$ and $\C_{G_1}^0 \cong \J'_{G_1}$. \end{proof}

\begin{theorem}[Jordan-H\" older theorem for weakly group-theoretical fusion categories]\label{main} Let $\C$ be a weakly group-theoretical fusion category. Then two composition series of $\C$ are equivalent. \end{theorem}

\begin{proof} Let \begin{equation}\label{cs-1}\vect = \C_0, \C_1, \dots , \C_n = \C,\end{equation} and \begin{equation}\label{cs-2}\vect = \D_0, \D_1, \dots , \D_m = \C,\end{equation} be two composition series of $\C$, with factors $G_1, \dots, G_n$, and $L_1, \dots, L_m$, respectively. Thus, 
for all $1 \leq i \leq n$, the Drinfeld center $\Z(\C_i)$ contains a Tannakian subcategory $\E_i \cong \Rep G_i$ such that $(\E_i')_{G_i}\cong \Z(\C_{i-1})$ as braided fusion categories, and for all $1 \leq j \leq m$, $\Z(\D_j)$ contains a Tannakian subcategory $\K_j \cong \Rep L_j$ and $(\K_j')_{L_j}\cong \Z(\D_{j-1})$ as braided fusion categories. The groups $G_1, \dots, G_n$ and $L_1, \dots, L_m$, are simple by assumption. 

We need to show that $n = m$ and the sequence $G_1, \dots, G_n$ is a permutation of the sequence $L_1, \dots, 
L_m$.  

\medbreak Let $G: = G_n$ and $L: = L_m$. So that $\E = \E_n \cong \Rep G$ and $\K = K_m \cong \Rep L$ are Tannakian subcategories of $\Z(\C)$ and 
the de-equivariantization $\Z(\C)_G$ (respectively, $\Z(\C)_L$) is a faithful $G$-graded extension of $\Z(\C)^0_G \cong \Z(\C_{n-1})$ (respectively, a faithful $L$-graded extension of $\Z(\C)^0_L \cong \Z(\D_{n-1})$).

\medbreak 
Since $G$ and $L$ are simple groups, then $\E$ and $\K$ contain no nontrivial proper fusion subcategories. Therefore we must have $\E\cap \K = \E$ or $\E\cap \K = \vect$. The first possibility implies that $\E = \K$. Hence $G \cong L$ and $\Z(\C_{n-1}) \cong \Z(\C)_G^0 = \Z(\C)_L^0 \cong \Z(\D_{m-1})$ as braided fusion categories. The results follows in this case by induction.

\medbreak We may thus assume that $\E\cap \K = \vect$. Let $F: \Z(\C) \to \Z(\C)_G$ be the canonical tensor functor. 
In view of  Proposition \ref{br-equiv}, $F$ induces by restriction an equivalence of fusion categories $F\vert_{\K}: \K \to F(\K)$. Hence $F(\K)$ contains no nontrivial proper fusion subcategories. Moreover, since $F(\K) \subseteq \Z(\C)_G$ is a $G$-stable fusion subcategory, it follows from Proposition \ref{simple-gstable} that either $F(\K) \subseteq \Z(\C)_G^0$ or $F(\K)$ is pointed of prime dimension and the group $G(F(\K))$ of its invertible objects is isomorphic to a normal subgroup of $G$. 

The last possibility implies that $G$ and $L$ are isomorphic cyclic groups of prime order. Note that, since $F(\K)$ is not contained in $\Z(\C)_G^0$, then $\K \cap \E' = \vect$. Lemma \ref{dimp} implies that $\Z(\C_{n-1}) \cong \Z(\C)_G^0 \cong \Z(\C)_L^0 \cong \Z(\D_{m-1})$ as braided fusion categories. Then the theorem follows by induction. 

\medbreak Let us consider the remaining possibility, namely, the case where $F(\K) \subseteq \Z(\C)_G^0$. In this case, Proposition \ref{br-equiv} implies that $\K \subseteq \E'$. 
Therefore the Tannakian subcategories $\E$ and $\K$ centralize each other in $\Z(\C)$. 
It follows from \cite[Proposition 7.7]{muegerII} that the fusion subcategory $\J = \E \vee \K$ generated by $\E$ and $\K$ is equivalent, as a braided fusion category, to the Deligne tensor product $\E \boxtimes \K$. We thus obtain that $\J \cong \Rep (G \times L)$ is a Tannakian subcategory of $\Z(\C)$ containing $\E$ and $\K$. 

\medbreak Since $\Z(\C_{n-1}) \cong \Z(\C)_G^0$ and $\Z(\D_{m-1}) \cong \Z(\C)_L^0$, it follows from Lemma \ref{ittd} that 
$\Z(\C)_G^0 \cong \Z(\C_{n-1})$ contains a Tannakian subcategory equivalent to $\J_G \cong \Rep L$,  $\Z(\C)_L^0 \cong \Z(\D_{m-1})$ contains a Tannakian subcategory equivalent to $\J_L \cong \Rep G$, and 
there are equivalences of braided fusion categories
$$\Z(\C_{n-1})_L^0 \cong \Z(\C)_{G\times L}^0  \cong \Z(\D_{m-1})_G^0.$$ Let $S_1, \dots, S_r$, $r \geq 1$, be the composition factors of a central series of $\Z(\C)_{G\times L}^0$.

An inductive argument applied, respectively, to the weakly group-theoreti\-cal fusion categories $\C_{n-1}$ and $\D_{m-1}$ implies on the one hand, that $r+1 = n-1$ and 
$L = L_m, S_1, \dots, S_r$ is a permutation of $G_1, \dots, G_{n-1}$, and on the other hand, that $r+1 = m-1$ and 
$G = G_n, S_1, \dots, S_r$ is a permutation of $L_1, \dots, L_{m-1}$. Hence $n = m$ and $G_1, \dots, G_{n-1}, G_n$ is a permutation of $L_1, \dots, L_{m-1}, L_m$. This finishes the proof of the theorem. 
\end{proof}

\section{Composition factors and length}\label{examples}

In view of Theorem \ref{main}, the composition factors and the length of a composition series are independent of the choice of the series. These will be called the \emph{composition factors} and the \emph{length} of $\C$, respectively. By construction, they are a Morita invariant of $\C$, that is, if $\D$ is a fusion category Morita equivalent to $\C$, then $\C$ and $\D$ have the same length and the same composition factors.

It is immediate from the definitions that a weakly group-theoretical fusion category is solvable if and only if its composition factors are cyclic groups of prime order.

\subsection{Group extensions and equivariantizations} Let $G$ be a finite group and let $\D$ be a group-theoretical fusion category. If $\C$ is a $G$-extension or a $G$-equivariantization of $\D$, then $\C$ is also weakly group-theoretical. 

Suppose that $G_1, \dots, G_n$ are the composition factors of $G$. It follows from Proposition \ref{normal-gding} that there is a series of fusion subcategories 
$$\D = \D_0 \subseteq \D_1 \subseteq \dots \subseteq \D_n = \C,$$  
such that $\D_i$ is a $G_i$-extension of $\D_{i-1}$, for all $i = 1, \dots, n$.

Therefore we obtain:

\begin{corollary} Let $G$ be a finite group with composition factors $G_1, \dots, G_n$ and let $\D$ be a group-theoretical fusion category with composition factors $S_1, \dots, S_m$. Suppose that $\C$ is a $G$-extension or a $G$-equivariantization of $\D$. Then the composition factors of $\C$ are $G_1, \dots, G_n, S_1, \dots, S_m$. In particular, the length of $\C$ equals the sum of the lengths of $\D$ and of $G$.
\end{corollary}

\begin{proof} In view of the argument above, it will be enough to show the statement on equivariantizations. So assume that the group $G$ acts on the tensor category $\D$ by tensor autoequivalences, so that $\C \cong \D^G$. 
Consider the crossed product (or semi-direct product) fusion category $\D \rtimes G$ constructed by Tambara in \cite{tambara}. The category $\D \rtimes G$ is a $G$-extension of $\D$. The category $\D$ has a natural structure of $\D \rtimes G$-module category such that $(\D \rtimes G)^*_\D \cong \D^G \cong \C$ \cite[Proposition 3.2]{nik}. Then $\C$ is Morita equivalent to $\D \rtimes G$ and therefore they have the same composition factors and the same length. This proves the corollary. \end{proof}

\subsection{Nilpotent semisimple Hopf algebras} A semisimple Hopf algebra $H$ is called \emph{nilpotent} if the category $\Rep H$ is nilpotent. 

Let $H$ be a semisimple Hopf algebra and let $\C = \Rep H$ be the fusion category of its finite dimensional representations. It follows from \cite[Theorem 3.8]{gel-nik} that  faithful gradings $\C = \oplus_{g \in G}\C_g$ of $\C$ correspond to central exact sequences of Hopf algebras $k \toto k^G \toto H \toto \overline H \toto k$, such that $\C_e = \Rep \overline H$.  

Combining this with the results in \cite{jh-hopf}, we obtain:

\begin{corollary} Let $G_1, \dots, G_n$ be the composition factors of the category $\Rep H$. Then the composition factors of the Hopf algebra $H$ are the dual group algebras $k^{G_1}, \dots, k^{G_n}$. \qed 
\end{corollary}

\subsection{Group-theoretical fusion categories}\label{gttic} Let $G$ be a finite group and let $\omega \in H^3(G, k^*)$. It follows from the discussion in Example \ref{pointed} that the composition factors of $\C(G, \omega)$ are the composition factors of $G$. 
Since the category $\Rep G$ is Morita equivalent to $\C(G, 1)$, then we obtain:

\begin{corollary} Let $G$ be a finite group. Then the composition factors of the category $\Rep G$ are exactly the composition factors of $G$ counted with multiplicities.

In particular, if $L$ is a group such that the categories $\Rep G$ and $\Rep L$ are Morita equivalent, then $G$ and $L$ have the same length and the same composition factors counted with  multiplicities. 
\end{corollary} 

\medbreak Furthermore, let $\C$ be a group-theoretical fusion category, so that $\C$ is Morita equivalent to a fusion category $\C(G, \omega)$. Then the composition factors of $\C$ are the exactly the composition factors of $G$.

\medbreak As a consequence of this, consider the case $\C = \Rep H$,  where $H = k^\Gamma {}^{\tau}\#_{\sigma} kF$ is a bicrossed product semisimple Hopf algebra associated to a matched pair of finite groups $(F, \Gamma)$, and $(\sigma, \tau) \in \Opext(k^\Gamma, kF)$ is a pair of compatible cocycles; see \cite{mk-ext}. In other words, $H$ is a Hopf algebra that fits into an abelian exact sequence $k \toto k^\Gamma \toto H \toto kF \toto k$.

\medbreak The matched pair $(F, \Gamma)$ has an associated group, denoted $F\bowtie \Gamma$, which is defined on the set $F \times \Gamma$ in terms of the compatible actions $\Gamma \overset{\triangleleft}\longleftarrow \Gamma \times F \overset{\triangleright}\longrightarrow \Gamma$ in the form
$$(x, s) (y, t) = (x (s \triangleright y), (s \triangleleft y) t), \qquad x, y \in F, \; s, t \in \Gamma.$$
Then $\C$ is Morita equivalent to the category $\C(F\bowtie \Gamma, \omega(\sigma, \tau))$, where  $\omega(\sigma, \tau)$ is the image of $(\sigma, \tau)$ under a map $\Opext(k^\Gamma, kF) \to H^3(k^\Gamma, kF)$ in a certain exact sequence due to G. I. Kac. See \cite{gp-ttic}.
This implies that the composition factors of $\C$ are the composition factors of $F \bowtie \Gamma$.

\medbreak The next example shows that for a semisimple Hopf algebra $H$, the composition factors of $H$ studied in \cite{jh-hopf} need not coincide with the group algebras or dual group algebras of the composition factors of the fusion category $\Rep H$.  Moreover, the length of $H$ may be different from the length of the group-theoretical fusion category $\Rep H$.

\begin{example} Recall from \cite[Example 4.7]{jh-hopf} that the composition factors of the Hopf algebra $H = k^\Gamma {}^{\tau}\#_{\sigma} kF$
are the (simple) Hopf algebras $kF_1, \dots, kF_r, k^{\Gamma_1}, \dots, k^{\Gamma_s}$, where $F_1, \dots, F_r$ are the composition factors of $F$ and $\Gamma_1, \dots, \Gamma_s$ are the composition factors of $\Gamma$.

\medbreak Let, for instance, $H = k^{\Ss_{n-1}} \# kC_n$, $n \geq 5$,  be the bicrossed product associated to the matched pair $(C_n, \Ss_{n-1})$ arising from the exact factorization $\Ss_{n} = \Ss_{n-1} C_n$ of the symmetric group $\Ss_n$,  where $C_n = \langle (12\dots n)\rangle$. See \cite[Section 8]{mk-ext}. So that the group $C_n \bowtie \Ss_{n-1}$ is isomorphic to $\Ss_n$. Hence the composition factors of the category $\Rep H$ are $\Aa_n$ and $\mathbb Z_2$, where $\Aa_{n}$ is the alternating group on $n$ letters. In particular, the length of $\Rep H$ is 2.

\medbreak On the other hand, the composition factors of the Hopf algebra $H$ are $k^{\Aa_{n-1}}, k^{\mathbb Z_2}, k\mathbb Z_{p_1}, \dots k\mathbb Z_{p_m}$, where $p_1, \dots, p_m$ are the prime factors of $n$ counted with multiplicities.
 \end{example}

\subsection{The Drinfeld center} Let $\C$ and $\D$ be weakly group-theoretical fusion categories. Then $\C \boxtimes \D$ is weakly group-theoretical. Suppose that  $G_1, \dots, G_n$ are the composition factors of $\C$ and $F_1, \dots, F_m$ are the composition factors of $\D$. Then the composition factors of $\C\boxtimes \D$ are $G_1, \dots, G_n, F_1, \dots, F_m$. So that the length of $\C \boxtimes \D$ equals the sum of the lengths of $\C$ and $\D$.

\begin{corollary} Let $\C$ be a weakly group-theoretical fusion category with composition factors $G_1, \dots, G_n$. Then the  composition factors of $\Z(\C)$ are $G_1, \dots, G_n, G_1, \dots, G_n$. In particular, the length of $\Z(\C)$ is twice the length of $\C$. \end{corollary} 

\begin{proof} The Drinfeld center $\Z(\C)$ is also a weakly group-theoretical fusion category. Moreover, $\Z(\C)$ is Morita equivalent to the tensor product $\C \boxtimes \C^{op}$, where $\C^{op}$ denotes the fusion category whose underlying $k$-linear category is $\C$ with opposite tensor product; see \cite{muegerII}, \cite{ostrik}. In addition, $\C^{op}$ is Morita equivalent to $\C$ with respect to the indecomposable module category $\M = \C$. \end{proof}

Let $G$ be a finite group and let $\omega$ be a 3-cocycle on $G$. Then the Drinfeld center of the pointed fusion category $\C = \C(G, \omega)$ is equivalent as a braided fusion category to the category of finite dimensional representations of the twisted Drinfeld double $D^\omega(G)$ studied in \cite{dpr}. See \cite{majid}.
As an application of the results above, we obtain a necessary condition for the representation categories of two twisted Drinfeld doubles to be equivalent:

\begin{corollary} Let $G$ and $L$ be finite groups and let $\omega_1 \in H^3(G, k^*)$, $\omega_2 \in H^3(L, k^*)$. Suppose that  $\Rep D^{\omega_1}(G)$ and $\Rep D^{\omega_2}(L)$ are equivalent as braided fusion categories. Then $G$ and $L$ have the same length and the same composition factors counted with  multiplicities.
\qed 
\end{corollary}

\bibliographystyle{amsalpha}

\end{document}